\documentclass[11pt]{article}

\usepackage[letterpaper, hmargin=1.3in, top=1in, bottom=1.3in, footskip=1in]{geometry}
\usepackage{float}
\usepackage{titlesec}
\titleformat{\subsection}[runin]{\normalfont\bfseries}{\thesubsection.}{.5em}{}[.]\titlespacing{\subsection}{0pt}{2ex plus .1ex minus .2ex}{.8em}
\titleformat{\subsubsection}[runin]{\normalfont\itshape}{\thesubsubsection.}{.3em}{}[.]\titlespacing{\subsubsection}{0pt}{1ex plus .1ex minus .2ex}{.5em}

\usepackage[labelfont=sc,font=small,labelsep=period]{caption}
\usepackage{amsmath} 
\usepackage{amssymb}
\usepackage{amsfonts}
\usepackage{latexsym}
\usepackage{amsthm}
\usepackage{amsxtra}
\usepackage{amscd}
\usepackage{bbm}
\usepackage{mathrsfs}
\usepackage{bm}
\usepackage{xcolor}
\usepackage{subcaption}
\usepackage[toc]{appendix}
\usepackage[utf8]{inputenc}

\usepackage{graphicx, color}

\definecolor{darkred}{rgb}{0.9,0,0.3}
\definecolor{darkblue}{rgb}{0,0.3,0.9}

\usepackage{ifthen}
\def\comment#1{\ifthenelse{\isodd{\value{page}}}{\marginpar{\raggedright\scriptsize{\textcolor{darkred}{#1}}}}{\marginpar{\raggedleft\scriptsize{\textcolor{darkred}{#1}}}}}  

\usepackage[nottoc,notlof,notlot]{tocbibind}
\usepackage{cite} 

\numberwithin{equation}{section}
\numberwithin{figure}{section}

\theoremstyle{plain} 
\newtheorem{theorem}{Theorem}[section]
\newtheorem*{theorem*}{Theorem}
\newtheorem{lemma}[theorem]{Lemma}
\newtheorem*{lemma*}{Lemma}
\newtheorem{corollary}[theorem]{Corollary}
\newtheorem*{corollary*}{Corollary}

\newtheorem*{proposition*}{Proposition}
\newtheorem{definition}[theorem]{Definition}
\newtheorem*{definition*}{Definition}
\newtheorem{conjecture}[theorem]{Conjecture}
\newtheorem*{conjecture*}{Conjecture}

\theoremstyle{definition} 

\newtheorem*{example*}{Example}
\newtheorem{remark}[theorem]{Remark}
\newtheorem*{remark*}{Remark}

\renewcommand{\leq}{\leqslant}
\renewcommand{\geq}{\geqslant}
\renewcommand{\epsilon}{\varepsilon}

\title{Convergence to closed-form distribution for the backward $SLE_{\kappa}$ at some random times and the phase transition at $\kappa=8$}
\author{Terry J. Lyons, Vlad Margarint and Sina Nejad}

\begin{document}

\maketitle

\begin{abstract}
We study a one-dimensional SDE that we obtain by performing a random time change of the backward Loewner dynamics in $\mathbb{H}$. The stationary measure for this SDE has a closed-form expression. We show the convergence towards its stationary measure for this SDE, in the sense of random ergodic averages. The precise formula of the density of the stationary law gives a phase transition at the value $\kappa=8$ from integrability to non-integrability, that happens at the same value of $\kappa$ as the change in behavior of the $SLE_{\kappa}$ trace from non-space filling to space-filling curve. Using convergence in total variation for the law of this diffusion towards stationarity, we identify families of random times on which the law of the arguments of points under the backward $SLE_{\kappa}$ flow converge to a closed form expression measure. For $\kappa=4,$ this gives precise characterization for the random times on which the law of the arguments of points under the backward $SLE_{\kappa}$ flow converge to the uniform law. 
\end{abstract}

\section{Introduction}

 The Schramm-Loewner Evolution, or $SLE_\kappa$ is a one parameter family of random planar fractal curves introduced by Schramm in \cite{schramm2000scaling}, that are proved to describe scaling limits of a number of discrete models that appear in planar Statistical Physics. For instance, it was proved in \cite{lawler2011conformal}  that the scaling limit of loop erased random walk (with the loops erased in a chronological order) converges in the scaling limit to $SLE_{\kappa}$ with  $\kappa = 2\,.$ Moreover, other two dimensional discrete models from Statistical Mechanics including Ising model cluster boundaries, Gaussian free field interfaces, percolation on the triangular lattice at critical probability, and Uniform spanning trees were proved to converge in the scaling limit to $SLE_{\kappa}$ for values of $\kappa=3,$ $\kappa=4,$ $\kappa=6$ and $\kappa=8$ respectively in the series of works \cite{smirnov2010conformal}, \cite{schramm2009contour}, \cite{smirnov2001critical}  and \cite{lawler2011conformal}. 
 In fact, the use of Loewner equation along with the techniques of stochastic calculus, provided tools to perform a rigorous analysis of the scaling limits of the discrete models. In this framework it has been established a precise meaning to the passage to the scaling limit and its conformal invariance. We refer to \cite{lawler2008conformally} and \cite{rohde2011basic} for a detailed study of the object and many of its properties.\\


 In the last years, questions concerning the behaviour of the SLE trace at the tip were asked in \cite{viklund2012almost}, where the almost sure multi-fractal spectrum of the SLE trace near its tip is computed, and in \cite{zhan2016ergodicity}, in which the ergodic properties of the tip of the SLE trace are computed in capacity time parametrisation.\\
In this paper, we perform a random time change in the context of the backward Loewner differential equation. This random time change was used for the forward Loewner differential equation in \cite{schramm2001percolation} to obtain the probability that the SLE trace passes to the left of a fixed point in the upper half plane. Compared with the approach in \cite{schramm2001percolation}, in our analysis we use it for the backward Loewner differential equation.

After performing the random time change, we obtain a one-dimensional diffusion, that describes, via a time change,  the cotangent of the argument of the points in the in $\mathbb{H}$ under the backward Loewner flow.
\begin{align}\label{SDE}
&dT_u=-4\frac{T_u}{1+\kappa T_u^2}du-dB_u , \hspace{5mm} T_0=0.
\end{align}

One of our results is the convergence in a random ergodic average sense of the law of the diffusion $T_u$ described by (\ref{SDE}) to its stationary measure. The density of this measure has the closed-form expression
$$\rho(T)=C\left(\kappa  T^2+1\right)^{-4/\kappa }\,,$$
where $C$ is a normalizing constant.

In the second part, we  use stronger convergence results for the diffusion process $T_u$ in order to identify a family of random times along the backward flow on which the distribution of the cotangent of the argument of points converges to the precise law given by the closed form expression from above.
This analysis is in the same spirit as a Skorokhod type Embedding Problem for the backward $SLE_{\kappa}$ flow. However, we emphasize that in the case of the Skorokhod Embedding Problem one is searching for integrable stopping times, compared to our case, when the sequence of stopping times goes a.s. to infinity.

In general, the time change that we work with provides a way to study small time behavior in the original time via large asymptotic behavior in the random time, for almost every Brownian path. Since there are many tools available for studying long time behavior of solutions to SDE's, we believe that this random time change can be used further in other questions related with the behavior of the backward $SLE_{\kappa}$ trace.

In the last section of the paper we use the scaling of the backward $SLE_{\kappa}$ maps to conjecture that the distribution that we have obtained is the same as the distribution of the tip of the $SLE_{\kappa}$ trace at any fixed non-zero time. 




\textbf{Acknowledgement:} The second author would like to acknowledge the support of ERC (Grant Agreement No.291244 Esig) between 2015-2017 at OMI Institute, EPSRC 1657722 between 2015-2018, Oxford Mathematical Department Grant and  the EPSRC Grant EP/M002896/1 between 2018-2019. Also, we would like to thank Christina Ziyan Zou, Stephen Muirhead, Aleksandar Mijatovic and Dmitry Belyaev for useful suggestions and reading previous versions of this manuscript. 
\section{Preliminaries}

Let $B_t$ be a standard one-dimensional Brownian motion. When studying the $SLE_{\kappa}\,,$ in the upper half-plane, we study the corresponding families of conformal maps  in the formats
\begin{enumerate} 
\item Partial differential equation version for the chordal $SLE_{\kappa}$ in the upper half-plane
\begin{equation}\label{4}
\partial_{t}f(t,z)=-\partial_{z}f(t,z)\frac{2}{z-\sqrt{\kappa}B_{t}}\,, \hspace{3mm} f(0,z)=z, z \in \mathbb{H}\,.
\end{equation}
\item Forward differential equation version for chordal $SLE_{\kappa}$ in the upper half-plane

\begin{equation}\label{5}
\partial_{t}g(t,z)=\frac{2}{g(t,z)-\sqrt{\kappa}B_{t}}\,, \hspace{10mm} g(0,z)=z, z \in \mathbb{H}\,.
\end{equation}

\item  Time reversal differential equation (backward) version for chordal $SLE_{\kappa}$ in the upper half-plane
\begin{equation}\label{6}
\partial_{t}h(t,z)=\frac{-2}{h(t,z)-\sqrt{\kappa}B_{t}}\,, \hspace{10mm} h(0,z)=z, z \in \mathbb{H}\,.
\end{equation}

\end{enumerate} 
There are connections between these three formulations for studying families of conformal maps. The solution to the equation (\ref{4}), i.e. the family of conformal maps satisfying (\ref{4}), is related with the family of conformal maps satisfying (\ref{5}), since by definition each instance of time $t\,,$ the map $z \to g_t(z)$ is the inverse of the map $z \to f_t(z)\,.$  In other words, the maps $f_t(z)$ ``grow'' the curve in the reference domain, while $g_t(z)$ maps conformally the slit domain obtained by the growing of the curve up to time $t$ to the reference domain. 
Moreover, we have the following lemma connecting the family of maps $z \to g_t(z)$ to $z \to h_t(z)$.

\begin{lemma}[Lemma 5.5 of \cite{kemppainen2017schramm}]
Let $h_t(z)$ be the solution to the backward Loewner differential equation with driving function $\sqrt{\kappa}B_t$ and let $f_t(z)$ be the solution of the partial differential equation version of the Loewner differential equation with the same driver. Then, for any $t \in \mathbb{R}_+$, the function $z \to f_t(z+\sqrt{\kappa}B_t)-\sqrt{\kappa}B_t$ and $z \to h_t(z)$ have the same distribution.
\end{lemma}

\section{A time change of Loewner differential equation and main result}

In this section, we consider a time change of the backward differential equation. The main reason for doing this is that when considering the backward Loewner differential equation in the upper half-plane and investigating its real and imaginary dynamics, we obtain naturally a system of coupled differential equations. These systems are usually hard to analyze. The purpose of the work in this section is to transform this coupled system of equations by making use of techniques of Stochastic Analysis into an one-dimensional radial-independent diffusion process that is easier to analyze. In turn, the imaginary part of the backward Loewner differential equation can be expressed in terms of this diffusion process making this approach a veritable change of coordinates.

We consider the probability space $(\Omega,\sigma(B_s, s \in [0,+\infty)), \mathbb{P})$, where $\mathbb{P}$ denotes the Wiener measure. On this space, we consider the standard Brownian motion $B_t$. Throughout the proof we use the L\' evy's characterization of Brownian motion for a local martingale that we obtain in our analysis, in order to obtain a new Brownian motion. 
 
Before going to the proof of the result in this section, we define a time change of the (backward) Loewner differential equation that is going to be the framework of study for our problem. First, we rewrite the equation 

\begin{equation}\label{6}
\partial_{t}h(t,z)=\frac{-2}{h(t,z)-\sqrt{\kappa}B_{t}}\,, \hspace{10mm} h(0,z)=z, \quad z \in \mathbb{H}\,,
\end{equation}
in the format 
\begin{equation}\label{6ztformat}
dz_t=\frac{-2}{z_t}dt-\sqrt{\kappa}dB_t\,, \hspace{10mm} z_0 \in \mathbb{H},
\end{equation}
where we make the identification $z_t=x_t+iy_t=h_t(z)-\sqrt{\kappa}B_t\,.$ Throughout this section, we choose $z_0=i$.
Furthermore, we split this flow in the upper half-plane into its real and imaginary components.
This gives 
\begin{equation}\label{xtandyt}
dx_t =\frac{-2 x_t}{x_t^2+y_t^2}dt-\sqrt{\kappa}dB_t\,, \hspace{5mm} dy_t=\frac{2y_t}{x_t^2+y_t^2}dt\,.
\end{equation}
We consider $D_t=x_t/y_t$. Applying It\^o's formula for the function $f(x,y)=\frac{x}{y}$, we obtain that
\begin{equation*}
dD_t=-\frac{\sqrt{\kappa}dB_t}{y_t}-\frac{4D_t}{x_t^2+y_t^2}dt\,.
\end{equation*}
We study the random time change given by
\begin{equation*}
\tilde{u}(s)=\int_0^{s}\frac{dt}{y_t^2}\,,
\end{equation*}
We consider also
\begin{equation}
c(u)=\inf\{ t \geq 0: \tilde{u}(s) \geq u\}\,.
\end{equation}
Alternatively, we can write $dc(u)=y^2_{c(u)}du\,.$\\
By setting 
\begin{equation}\label{timechangedBM}
\tilde{B}(u)=\int_{0}^u\frac{dB_{c(r)}}{y_{c(r)}}\,,
\end{equation}
we obtain that $T_u$ defined as $T_u=\frac{D_{c(u)}}{\sqrt{\kappa}}, u \geq 0$ satisfies the following stochastic differential equation 
\begin{equation}
dT_u=-4\frac{T_u}{1+\kappa T^2_u}du +d\tilde{B}_u\,.
\end{equation}

The continuous local martingale $\tilde{B}_u$ has quadratic variation $u$. Thus, by L\' evy's characterization, we obtain that  $\tilde{B}_u$ is a Brownian motion. Thus, when performing this random time change we obtain a one-dimensional stochastic differential equation. We emphasize that the filtration that we consider when performing the random time change it is changing from $\mathcal{F}_t$ ( the natural filtration of the standard Brownian motion up to time $t \in [0, +\infty)$)  to $\mathcal{F}_{\tilde{u}(t)},$ where $\tilde{u}(t)$ is the random time change described before.

We introduce next a list of useful definitions for the work in this section:

\begin{definition}[Stationary distribution]
Let $X_t$ be a one-dimensional Markov process. A stationary distribution for $X_t$ is a probability distribution $\psi$ on $\mathbb{R}$ such that if $X_0 \sim \psi$, then $X_t$ has distribution $\psi$ for all $t \geq 0$. 
\end{definition}

In particular, one can obtain from this definition that 
$$\frac{d}{dt}\mathbb{E}[f(X_t)|X_0 \sim \psi]=0,$$
with $f \in L^1(\mu)$, where $\mu(dx)$ is the stationary distribution of $X_t$.

\begin{definition}
For a one-dimensional diffusion $X_t$, let $P_tf(x)=\mathbb{E}^x[f(X_t)]$ be its transition semigroup. A one-dimensional diffusion is said to converge strongly to its invariant measure $\mu$ if for any compact $K \in \mathbb R$, for $f \in L^{\infty}(\mathbb R)$ we have that
$sup_{x \in K}|P_tf(x)-\mu(f)| \to 0, \hspace{2mm}$ as $t \to \infty,$ where $\mu(f)=\int_{\mathbb R }f(y)\mu(dy)$.
\end{definition}
The stationary distribution of a diffusion is obtained using the Kolmogorov forward equation that we introduce next.
\begin{definition}[Kolmogorov forward equation]\label{kolmogorovf}
Let $C_0^2(\mathbb R) $ be the space of differentiable functions with continuous derivatives up to the second order and that vanish at infinity. Let $X_t$ be a diffusion on $\mathbb R$ whose transition semigroup admits a density, with generator $$Af(y)=a(y)\frac{\partial ^2f(y)}{\partial y} +b(y)\frac{\partial f(y)}{\partial y}.$$ Let $p_t(x,y)$ be the density of the transition semigroup (i.e. the function $p_t(x,y)$ in $P_tf(x)=\mathbb E^x[f(X_t)]=\int_{\mathbb R} f(y)p_t(x,y)dy$, $f \in C_0^2(\mathbb R)$).
Then $p_t(x,y)$ satisfies the Kolmogorov forward equation
$$\frac{\partial p_t(x,y)}{\partial t}=\frac{\partial^2(a(y)p_t(x,y))}{\partial y}+\frac{\partial(b(y)p_t(x,y))}{\partial y}.$$
\end{definition}
In order to compute the stationary distribution, we will consider the solution $p_t(x,y)$ of the Kolmogorov forward equation for which $\frac{\partial p_t(x,y)}{\partial t}=0$.
Next, we introduce a special class of functions that will appear in our analysis.
\begin{definition}
For $z \in \mathbb{D}$ and for $\alpha, \beta , \gamma\in \mathbb{C}$ with $\gamma$ not a nonpositive integer, the hypergeometric function is defined as the following convergent series 
$$_2F_1(\alpha,\beta,\gamma;z)=1+\sum_{k=1}^{\infty}\frac{(\alpha)_k(\beta)_k}{(\gamma)_kk!}z^k,$$
where $(\cdot)_k$  is the rising  Pochhammer symbol, which is defined by:
\begin{equation}
(\cdot)_k=\begin{cases}
     1, & \text{if}\hspace{2mm}  k=0,\\
     \frac{\Gamma(\cdot+k)}{\Gamma(\cdot)}, & \text{if} \hspace{2mm} k>0.
   \end{cases}
\end{equation}
\end{definition}
\begin{remark}
It can be shown that for complex arguments $z$ with $|z|>1$ the function can be analytically continued along any path in the complex plane that avoids the branching points $1$ and infinity. These functions are extremely well studied (see \cite{abramowitz1965handbook}).
\end{remark}
\begin{remark}
Let us consider $\alpha, \beta, \gamma$ as before. The hypergeometric  functions appear as solutions to the hypergeometric equations of the form
\begin{equation}
z(1-z)\phi^{''}(z)+[\gamma-(\alpha+\beta+1)z]\phi'(x)-\alpha\beta\phi(z)=0
\end{equation}
for $z \in \mathbb{C}\setminus \{1\}.$
\end{remark}
In this section, we will study differential equations that are related with the hypergeometric differential equations. The solutions of these equations can be written in terms of hypergeometric functions. 



\color{black}

We are now ready to state the main result. We emphasize that in the text of the main Theorem and throughout the paper, we use the notation $ctg(\theta)$ for the cotangent of the angle $\theta.$
\begin{theorem}\label{result1}
For $\kappa>0$, let us consider the random time change $$\tilde{u}(S, \omega)=\int_0^S\frac{dt}{y_t^2(\omega)},$$ where $y_t$ is the imaginary part of the chordal backward $SLE_{\kappa}$ in \eqref{xtandyt}. 
We consider the process $T_u$ started with initial value $T_0=0$ given by the stochastic differential equation 
\begin{equation}
dT_u=-4\frac{T_u}{1+\kappa T^2_u}du +d\tilde{B}_u\,,
\end{equation}
where $\tilde{B}_u$ is given by \eqref{timechangedBM}.
Then, this process has a stationary distribution with density 
$$\rho(T)=C\left(\kappa  T^2+1\right)^{-4/\kappa }\,,$$
where $C$ is a normalizing constant.
Let us consider the functions $f \in L^{1}(\mu)$, where $\mu(dx)=\frac{dx}{(1+x^2)^{4/\kappa}}$.\\
For $\kappa <8$, the process $T_u$ converge in the sense of random ergodic averages towards its stationary law, i.e.
$$\left|\frac{1}{\tilde{u}(S, \omega)}\int_{0}^{\tilde{u}(S, \omega)}f(T_s(\omega))ds-\mu(f)\right| \xrightarrow{ S \to \infty }0$$ almost surely, with respect to the Wiener measure.\\
Let $(a_n)$ be a strictly increasing sequence such that $a_n \to \infty$ as $n \to \infty$. Moreover,  let
$$s_n(\omega)=\inf \left\lbrace t \in [0, \infty): \int_0^t\frac{ds}{y_s^2}(\omega)=a_n \right\rbrace.$$
Then for any bounded measurable function $f \in L^{\infty}$ for $\kappa <8$, we obtain $$f(ctg(\arg(h_{s_n(\omega)}(i))))=f(T_{a_n})$$ and as $n \to \infty$ we have that the law of $f(T_{a_n})$ converges strongly to $\mu(f)$, i.e. $$f(ctg(\arg(h_{s_n(\omega)}(i)))) \to \mu(f).$$
\end{theorem}

\begin{figure}\label{F1}
\begin{center}
\includegraphics[scale=0.4]{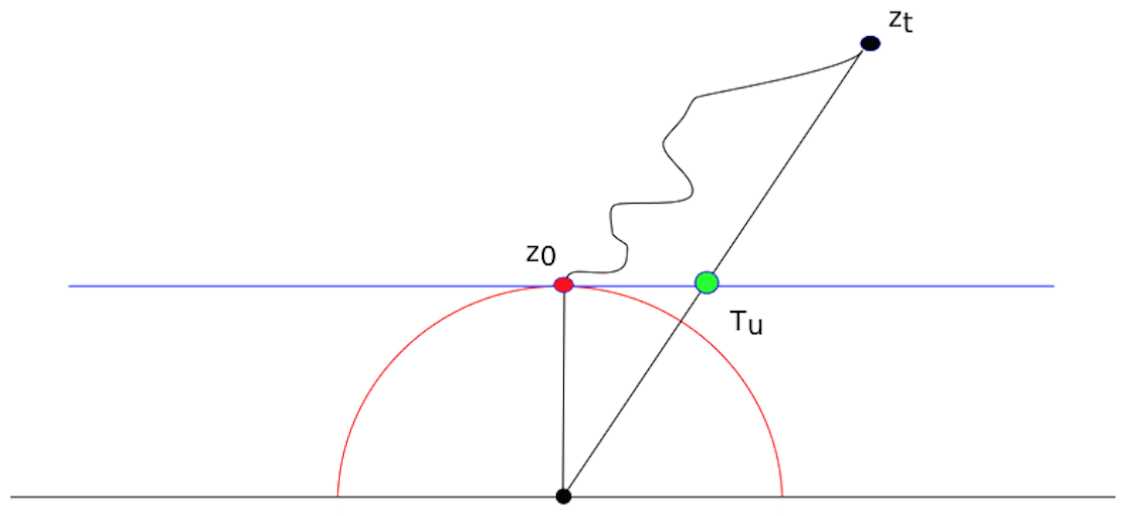}
\caption{The process $T_u$ is equivalent to a diffusion on the horizontal constant height line $y=1$ in $\mathbb{H}$.}
\end{center}
\end{figure}

\begin{corollary}
For $\kappa=4$, on the sequence of random times $s_n(\omega)$ the law of the  backward $SLE_{\kappa}$ flow converges to the uniform distribution. Indeed, changing the coordinates from cotangent of the argument to the argument, we obtain that the density of the stationary measure can be written as $d\tilde{\rho}(\theta)= \tilde{C} \sin ^{\frac{8}{\kappa}-2}(\theta)d\theta\,,$  for some normalizing constant $\tilde{C}$.
\end{corollary}

\begin{remark}
From the definition of our process, we have that $D_{c(u)}=T_u$ and $D_{S}= T_{\tilde{u}(S)}$, where $S$ is a deterministic time. Thus, we consider the format of the random ergodic average as in the main theorem, in order to relate it with the backward Loewner flow up to a fixed time.
\end{remark}

\begin{remark}
We remark that there is a phase transition at the value $\kappa=8$ in terms of the integrability of the density of the stationary distribution of the process $T_u$, i.e. $T_u$ has stationary distribution that is a probability measure for all $\kappa <8$ and is not a finite measure  for all $\kappa \geq 8\,.$ 

\end{remark}

The proof of the main result is divided in several subsections. First, we prove a useful lemma in the next section.

\section{Technical lemma}

We recall that we start with $$h_{t}(i)-\sqrt{\kappa}B_{t}=x_{t}(i)+iy_{t}(i)\,,$$ and we consider the process $D_{t}=\frac{x_{t}(i)}{y_{t}(i)}\,,$ for $t \in [0, +\infty)$.

By It\^{o}'s formula, we have that $\frac{x_{t}(i)}{\sqrt{\kappa}y_{t}(i)}$, as $t$  varies in $[0, +\infty)$, satisfies the SDE
\begin{align}\label{first}
d \frac{x_{t}(i)}{y_t(i)}&=-4\frac{x_{t}(i)/y_t(i)}{1+\kappa\left(x_t(i)/y_t(i)\right)^2} \frac{dt}{y^2_t} + \frac{dB_t}{y_t}.
\end{align}

\noindent\\
Thus, we can rephrase (\ref{first}) as before, to the following SDE 
\begin{align}\label{refsede}
dT_u&=-4\frac{T_u}{1+\kappa T_u^2}du+d\tilde{B}_u,\nonumber\\
T_0&=0.
\end{align}
The starting point of the diffusion is $T_0=0$ since $\hat{h}_0(i)=i\,.$

We consider the following useful lemma that gives one of the properties of the random time change $\tilde{u}(S, \omega)$ defined in the previous section.

\begin{lemma}\label{timechange}
Let $\kappa>0$. The random time change $\tilde{u}(S, \omega)$ is bounded from below, for all  $S>0\,,$ uniformly in $\omega$,  by
$$ \tilde{u}(S,\omega) \geq \frac{\kappa}{4}\log \left(1+\frac{4S}{\kappa} \right)\,.$$
\end{lemma}

\begin{proof}
Using the equation for the dynamics of $y_t$ from the Loewner differential equation \eqref{xtandyt}, we obtain that $\partial_t[y_t(z)^2]= 2y_t(z)\partial_t y_t(z)=\frac{4}{\kappa}\frac{y_t(z)^2}{|z_t(z)|^2} \leq \frac{4}{\kappa}$. Thus, we obtain that

$$y_t^2 \leq y_0^2+\frac{4t}{\kappa}\,,$$
i.e.
$$\frac{dt}{y_t^2} \geq \frac{dt}{(y_0^2+4t/\kappa)}.$$ By integrating up to time $S$,  we obtain that
$$\int_0^S \frac{dt}{y_t^2} \geq \int_0^S \frac{dt}{(y_0^2+4t/\kappa)} =\frac{\kappa}{4} \log \left(1+\frac{4S}{\kappa y_0^2}\right)\,.$$

\end{proof}

\section{The stationary measure of the diffusion process $T_u$ }
In this section, we return to studying
 \begin{equation}\label{SDE1}
dT_u=-4\frac{T_u}{1+\kappa T^2_u}du +d\tilde{B}_u\,.
\end{equation}
We recall that from Lemma \ref{timechange}, we obtain that $\tilde{u}(S,\omega) \geq \log(1+4S/\kappa)$, for almost all realizations of the Brownian motion. In order to link the behavior of the backward Loewner flow and the behavior of the law of solutions of the SDE \eqref{SDE1}, we need to study $\tilde{u}(S, \omega)$ as $S \to \infty$. Thus, questions about the long term behavior of the SDE (\ref{SDE1}) become natural. In this section, we show that there is an explicit stationary law for the SDE (\ref{SDE1}).

The density of the stationary measure for the process $T_u$ can be obtained as a time-independent solution of the Kolmogorov forward equation (see Definition \ref{kolmogorovf}). For the process $T_u$, we obtain


$$ \frac{1}{2} \rho''(T)+ \frac{4 T}{\kappa  T^2+1} \rho'(T)+\frac{\left(4-4 \kappa  T^2\right) }{\left(\kappa  T^2+1\right)^2}\rho(T)=0.$$
The solution to the above equation process  is $$ \rho(T)= C_1 T \left(\kappa  T^2+1\right)^{-4/\kappa } \, _2F_1\left(\frac{1}{2},-\frac{4}{\kappa };\frac{3}{2};-T^2 \kappa \right)+2 C_2 \left(\kappa  T^2+1\right)^{-4/\kappa }\,,$$
where $C_1$ and $C_2$ are constants to be chosen.

\vspace{4mm}

\begin{figure}[H]
\begin{center}
\includegraphics[scale=0.5]{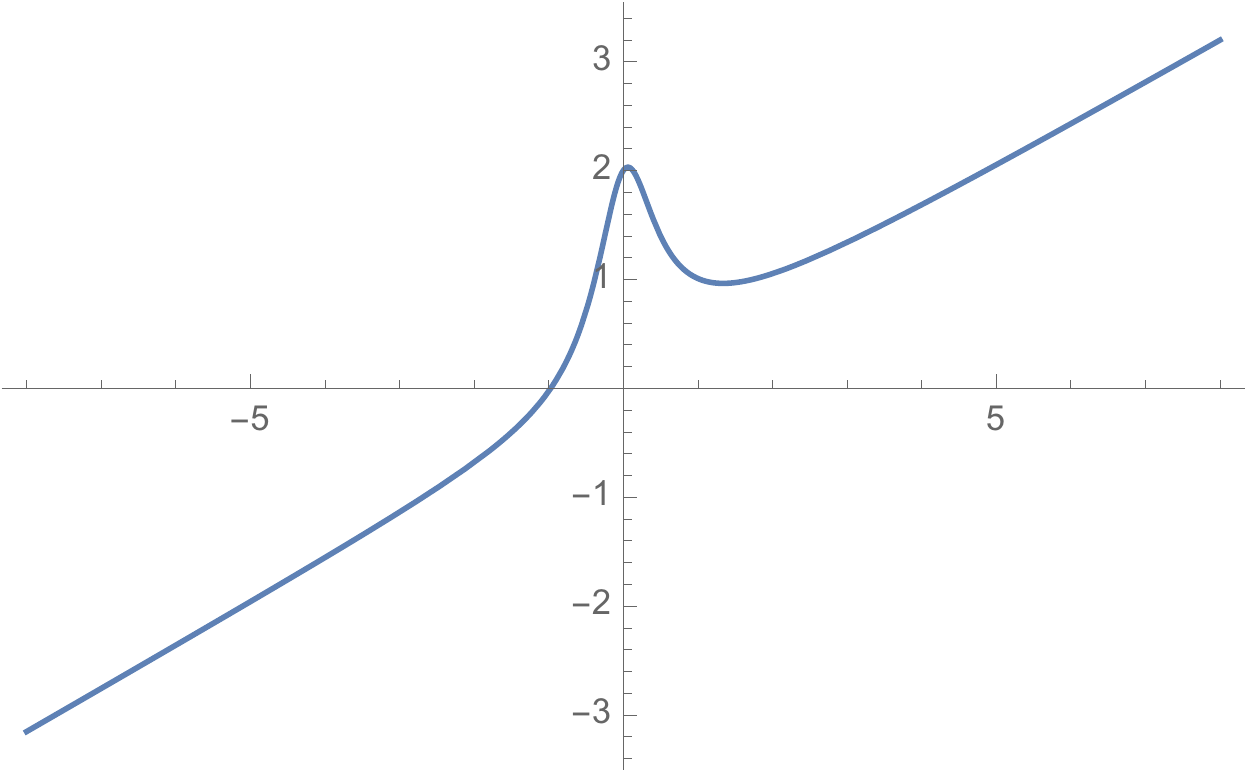}
\caption{Solution of the Kolmogorov forward equation for $T_u$ for $C_1=C_2=1$ and $\kappa=1.2$\,.}
\end{center}
\end{figure}

 We observe that the explicit solution contains a function of the form $$T \left(\kappa  T^2+1\right)^{-4/\kappa } \, _2F_1\left(\frac{1}{2},-\frac{4}{\kappa };\frac{3}{2};-T^2 \kappa \right).$$ 
 
Next, we study the behavior of the Hypergeometric function near its singularities. We use the following formula 
(see (15.3.7) in \cite{abramowitz1965handbook})

\begin{align*}
_2F_1\left(a, b , c ; z\right) &=\frac{\Gamma(b-a)\Gamma(c)}{\Gamma(b)\Gamma(c-a)}(-z)^a \hspace{1mm} _2F_1\left(a, a-c+1, a-b+1; \frac{1}{z}\right)\nonumber\\
&+\frac{\Gamma(a-b)\Gamma(c)}{\Gamma(a)\Gamma(c-b)}  (-z)^b \hspace{1mm} _2F_1\left(b, b-c+1, -a+b+1; \frac{1}{z}\right)
\end{align*} 
given that $a-b \not \in \mathbb{Z}$ and $z \notin (0,1)$. 

Since $T \in \mathbb{R}$, we restrict the functions to real variables. Thus, we have that
\begin{align}\label{formula1hypergeom}
_2F_1\left(a, b, c ; x\right) &=\frac{\Gamma(b-a)\Gamma(c)}{\Gamma(b)\Gamma(c-a)}(-x)^a \hspace{1mm} _2F_1\left(a, a-c+1, a-b+1; \frac{1}{x}\right)\nonumber\\
&+\frac{\Gamma(a-b)\Gamma(c)}{\Gamma(a)\Gamma(c-b)} (-x)^b\hspace{1mm} _2F_1\left(b, b-c+1, -a+b+1; \frac{1}{x}\right).
\end{align} 

For $b-a$ a nonnegative integer, doing the interchange $a$ with $b$ and keeping for simplicity the notation $b=a+n$, we have that (see (15.3.13) in \cite{abramowitz1965handbook}).
\begin{align}\label{formulahypergoem2}
&_2F_1\left(a, b, c ; z\right)=\frac{(-z)^{-a}}{\Gamma(a+n)} \sum_{k=0}^{n-1}\frac{(a)_k(n-k-1)!}{k!\Gamma(c-a-k)}z^{-k}\nonumber\\&+
\frac{(-z)^{-a}}{\Gamma(a)}\sum_{k=0}^\infty\frac{(a+n)_k}{k!(k+n)!\Gamma(c-a-k-n)}(-1)^k z^{-k-n}(\ln(-z)+\Psi(k,n,a)),
\end{align}
where $\Psi(k,n,a)=\psi(k+1)+\psi(k+n+1)-\psi(a+k+n)-\psi(c-a-k-n)$ and $\psi(z)=\frac{\Gamma'(z)}{\Gamma(z)}$, for $z \in \mathbb{C} \setminus \{0, -1, -2,...\}$.
The same holds when we restrict the functions to real variables. 

For $a-b \neq \mathbb{Z}$, we apply the formula \eqref{formula1hypergeom} for $T \to \infty$.
Let us consider $m=\frac{4}{\kappa}.$ Using the series representation of the hypergeometric function, we obtain that for a function $f(T)=_2F_1(\frac{1}{2}, -m, \frac{3}{2}; -T^2\kappa)$ as $T \to 0$, we have that  $$f(T)=\sum_{k=0}^{\infty}\frac{(-m)_k(-1)^kT^{2k}}{k!(2k+1)}.$$ In particular, the leading term is a constant. We use this and the formula \eqref{formula1hypergeom} to study the asymptotic behavior of $f(T)$ as $T \to \infty.$

Since in general for $p,q$ in $_2F_1(p, q, r ; x)$ are critical exponents at infinity, the asymptotic behavior as $T \to \infty$ of the function $$f(T)=_2F_1(\frac{1}{2}, -m, \frac{3}{2}; -T^2\kappa)$$ depends if $\operatorname{Re}m >-\frac{1}{2}$ or $\operatorname{Re}m<-\frac{1}{2}$.
In our case, $\operatorname{Re}m=m=\frac{4}{\kappa}>-\frac{1}{2}$. Thus, as $T \to \infty$, we obtain that $$f(T)=\frac{1}{2m+1}.$$
Finally, for $\kappa>0$, we have $$\lim_{T \to \infty} T \left(\kappa  T^2+1\right)^{-4/\kappa } \, _2F_1\left(\frac{1}{2},-\frac{4}{\kappa }, \frac{3}{2};-T^2 \kappa \right) \sim \frac{T}{1+8\kappa},$$
and 
$$\lim_{T \to -\infty} T \left(\kappa  T^2+1\right)^{-4/\kappa } \, _2F_1\left(\frac{1}{2},-\frac{4}{\kappa }, \frac{3}{2};-T^2 \kappa \right) \sim \frac{T}{1+8\kappa}.$$
Since we are searching for the solution that is a density of a probability distribution, the first constant needs to be zero, since otherwise the solution will also contain negative values.

For $a-b \in \mathbb{Z}$, we apply the formula \eqref{formulahypergoem2} and as $T \to -\infty$, we have that  
$$\lim_{T \to -\infty} T \left(\kappa  T^2+1\right)^{-4/\kappa } \, _2F_1\left(\frac{1}{2},-\frac{4}{\kappa }, \frac{3}{2};-T^2 \kappa \right) \sim T \frac{ \log(\kappa T^2)+C}{1+8\kappa},$$
where $C$ is a constant obtained from the formula \eqref{formulahypergoem2}. Thus, the constant $C_1$ should be zero, since otherwise, as before, we obtain a negative solution that is not a probability density.

\color{black}
Finally, we obtain an integrable probability density on $\mathbb{R}$ for $\kappa \in [0, 8)$ that has the form
$$\rho(T)=C\left(\kappa  T^2+1\right)^{-4/\kappa }\,.$$
where $C=\left(\int_{-\infty}^{+\infty}(\kappa T^2+1)^{-4/\kappa}dT\right)^{-1}\,.$

\color{black}
\begin{figure}[H]
\begin{center}
\includegraphics[scale=0.5]{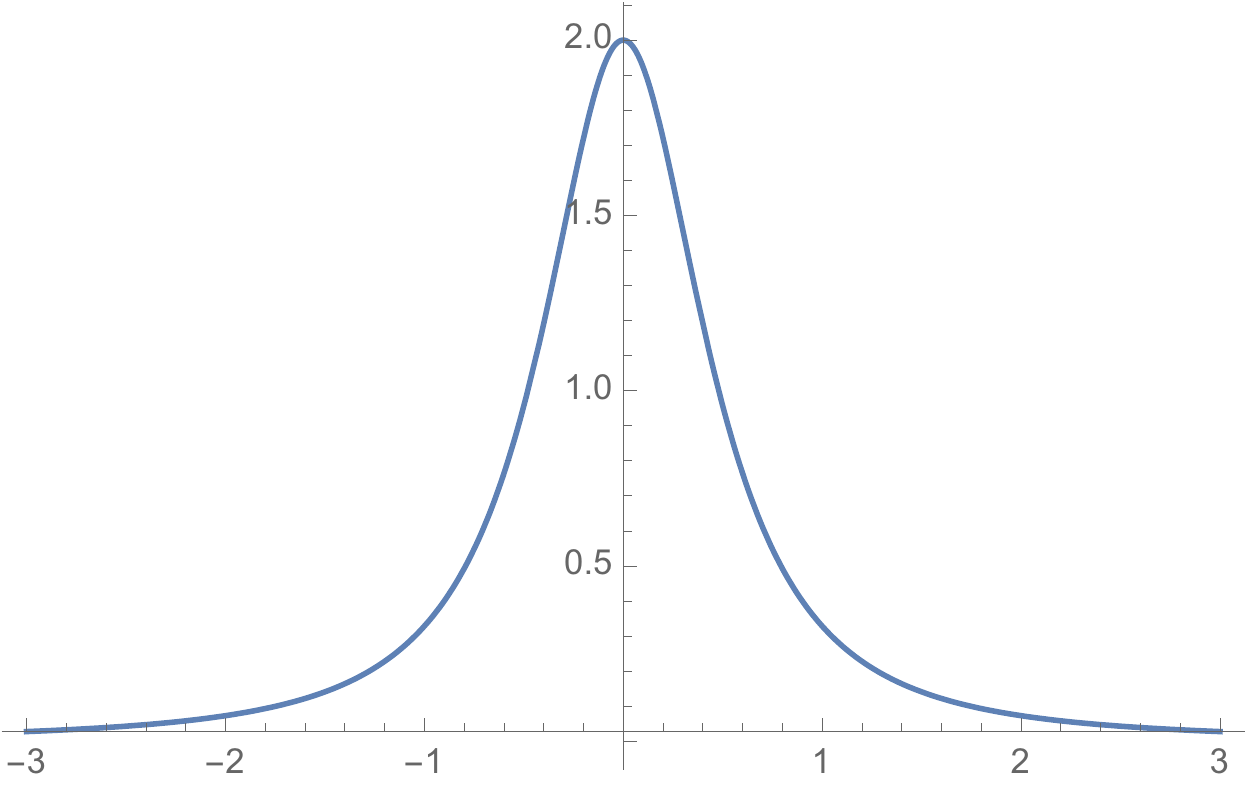}
\caption{The stationary measure for the process $T_u$ for $\kappa=3$\,.}
\end{center}
\end{figure}

\begin{remark}
Changing the coordinates from cotangent of the argument to the argument, we obtain that the density of the stationary measure can be written as $d\tilde{\rho}(\theta)= \tilde{C} \sin ^{\frac{8}{\kappa}-2}(\theta)d\theta\,.$ 
This function changes its behaviour at $\kappa=4$, being concave for $\kappa<4$, while being convex for $\kappa >4$. At $\kappa=4$ the function is constant. Thus, for $\kappa=4$, we obtain a uniform distribution on $[0,\pi]\,.$
\end{remark}

\begin{figure}[H]
\begin{center}
\includegraphics[scale=0.5]{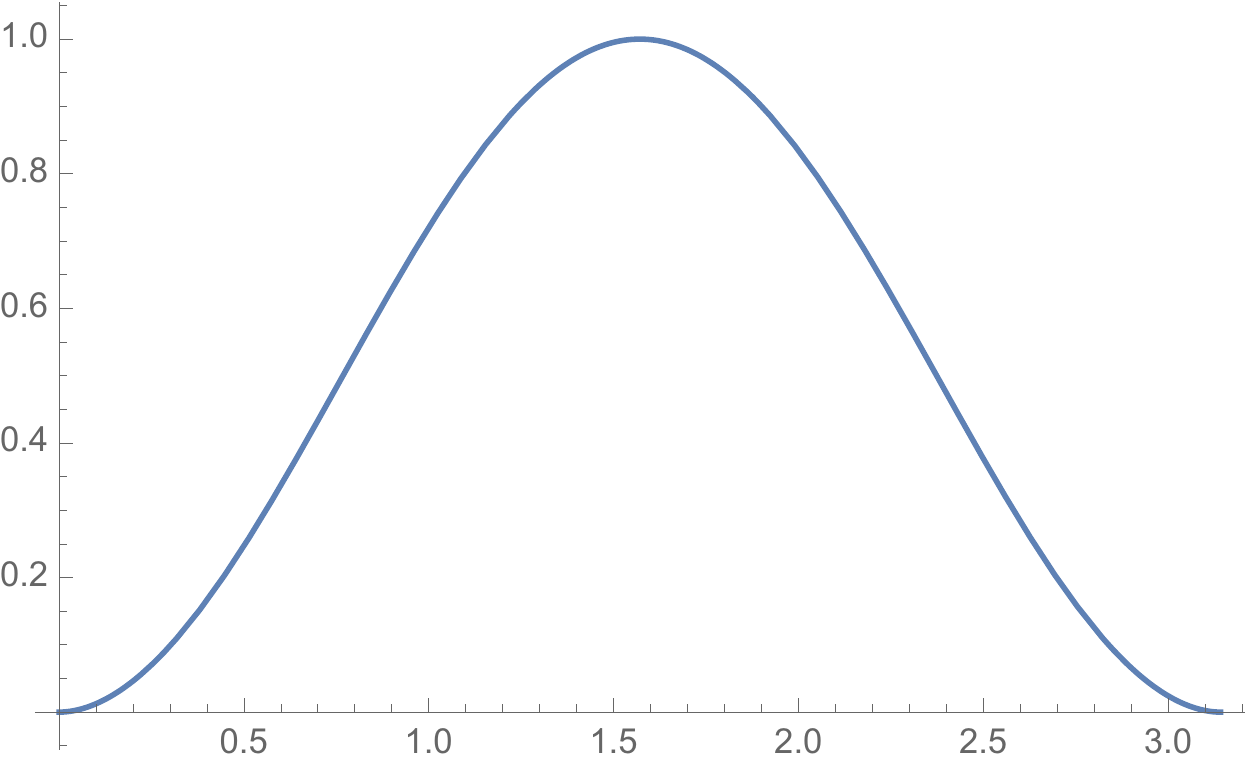}
\caption{Density of the stationary measure for the argument via change of variable $ctg (\theta) \to \theta$ from the stationary measure of $T_u$ for $\kappa =2.3$.}
\end{center}
\end{figure}

\begin{figure}[H]
\begin{center}
\includegraphics[scale=0.5]{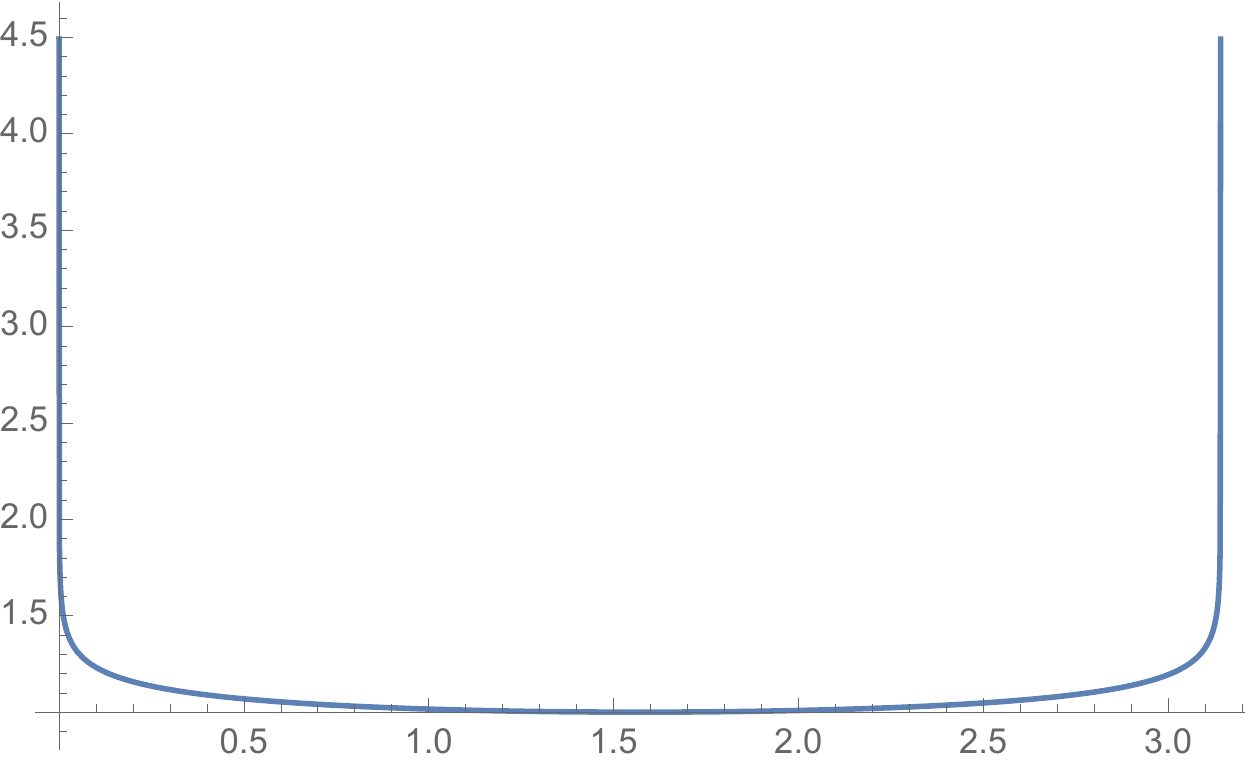}
\caption{Density of the stationary measure for the argument via change of variable $ctg(\theta) \to \theta$ from the stationary measure of $T_u$ for $\kappa =4.5$.}
\end{center}
\end{figure}

Comparing the two pictures from above, we see that for $\kappa>4$, the change in convexity of the density of the stationary measure indicates that more mass is given to the the values of the argument $\theta$ closer to $0$ and $\pi$ than to the ones closer to $\pi/2$ as in the case $\kappa<4$.

\section{Proof of the main result}
In this section, we give the proof of our main result. The analysis consists of two parts. First, we use a result from \cite{rogers2000diffusions} that gives convergence to the stationary distribution of the laws of a one-dimensional diffusion as time goes to infinity, in the sense of ergodic averages, for almost every Brownian path. We stress that this result gives the convergence towards the stationary law for large deterministic times.

We use that for a diffusion $X_t$ with drift $\mu$ and variance $\sigma^2$, the scale function can be computed via $s(x)=\int_{x_0}^x\exp(-\int_\eta^y\frac{2\mu(z)}{\sigma^2(z)}dz)dy$, where $x_0$, $\eta$ are points fixed (arbitrarily) in $(a,b)$. 
\begin{definition}
The function $m(\zeta)=\frac{1}{\sigma^2(\zeta)s'(\zeta)}$ is the density of the speed measure. 
\end{definition}
For the process $T_u$, we obtain the following:
$$s(x)=e^{-2\int_0^x\frac{-4/\kappa u}{1+u^2}du}=e^{8/\kappa \int_0^x \frac{u}{1+u^2}du}=e^{4/\kappa\int_1^{1+x^2}\frac{1}{l}dl}=e^{4/\kappa \log(1+x^2)}=(1+x^2)^{4/\kappa}.$$
Using the definition of the speed measure we obtain that $m(x)=\frac{2}{(1+x^2)^{4/\kappa}}$.
The first result that we use is the following theorem, that gives the a.s.\ convergence of the ergodic averages up to a deterministic time. Using the previous computations, we obtain that the process $T_u$ satisfies the conditions of Theorem \ref{Rw} for $\kappa < 8\,.$

\begin{theorem}[Theorem $53.1$, \cite{rogers2000diffusions}]\label{Rw}
Let $T_s$ be a recurrent diffusion on the real line in natural scale, such that $\int_{\mathbb{R}}\mu(dx) <\infty $, where $\mu(dx)$ is the stationary measure of the diffusion, that is up to a constant equal to the speed measure. Let $f \in L^1(\mu)$ and $f \geq 0$, Then 
$$ \frac{1}{t}\int_0^tf(T_s)ds \xrightarrow{a.s.} \mu(f),$$
as $t \to \infty.$
\end{theorem}
\begin{proof}[Proof of the first part of the main result]
Let $Z_u(f):=\frac{1}{u}\int_0^u f(T_s)ds.$ By Theorem \ref{Rw}, we have that for a.e. $\omega$ and $\forall \epsilon>0$, $\exists R(\omega)>0$ such that $$ \sup_{u>R(\omega)}|Z_{u}(\omega)-\mu(f)|\leq \epsilon.$$

Since the convergence in Theorem \ref{Rw} hols a.s., we have the following: For almost every sample $\omega$ of the Brownian motion, the limit $\lim_{u \to \infty} Z_u \to \mu (f)$. Since this is a a.s. result about the limit it holds on every sequence of times that go to infinity. Pick the times $\tilde{u}(S, \omega)$. We have that from Lemma \ref{timechange}, that a.s. $\tilde{u}(S, \omega) \to \infty $ as $ S \to \infty$. Then a.s., $Z_{\tilde{u}(S, \omega)}  \to \mu(f)$ as $ S \to \infty.$
\color{black}
\end{proof}

\section{The phase transition at $\kappa=8$ and the law of the backward Loewner flow at asymptotic large times}

In this section we show how stronger notions of convergence (convergence in total variation norm) for Stochastic Differential Equations towards their stationary measure give information about the law of the backward SLE trace at certain times. For $\kappa=4$, these times are exactly the times when the law of $\operatorname{Arg}(h_{t}(i)-\sqrt{\kappa}B_t)$ converges to the uniform distribution. 

In the setting of the Skorokhod Embedding Problem, given a stochastic process and a fixed distribution, an integrable stopping time should be constructed in order to obtain the fixed distribution when evaluating the process at this stopping time.  In this section, we discuss a similar problem in the context of Loewner differential equation. We show how to construct random times on which the law of the argument of points under this flow converge to the uniform law for $\kappa=4$. For more details on Skorokhod problems for other classical processes, see \cite{obloj2004skorokhod}. Moreover, we show that this belongs to a general family of laws that are obtained in the same manner for the $SLE_{\kappa}$ trace for other values of $\kappa <8$.

Let us consider the dynamics of the imaginary value of the backward Loewner differential equation. The random time change that we study is a natural process that is written in terms of the imaginary value of the backward Loewner flow. Moreover, it is a stochastic process that is a.s.\ increasing. Thus, it is natural  to consider the hitting times of constant levels of this process.

For the diffusion process $T_u$ defined by the SDE (\ref{SDE1}) there are stronger convergence results towards the stationary measure (at sequences of large deterministic times) than the convergence in the sense of ergodic averages.

We use the main result from \cite{ganidis1999convergence}, that we state below. Let us consider the following solution of a one-dimensional SDE: $$X_t = x + B_t -\frac{1}{2}\int_0^t b(X_s)ds,$$ 
where $b:\mathbb R \to \mathbb R$ is a continuous function. When imposing conditions on the behavior of the function  $b(x)$, we obtain various convergence results for the law of the stochastic process $X_t$ towards its stationary measure in stronger senses than the ergodic convergence.

It is proved in \cite{ganidis1999convergence} that for one-dimensional diffusion processes,  if $|b(x)| \sim \frac{C}{x}$ with $C>1$, as $|x| \to \infty$, then we have that there exists $\gamma>0$ and $t_0>0$ such that for any compact set $K$, there exists $C(K)>0$ such that for any $f \in L^{\infty}$ and $t \geq t_0$, we have
\begin{equation}
\sup_{x \in K}|P_tf(x)-\mu(f)| \leq C(K) \frac{1}{t^{\gamma}}||f||_{\infty}\,,
\end{equation}
where $P_t$ is the transition semigroup of the one-dimensional diffusion $X_t$.

\begin{proof}[Proof of the second part of Theorem \ref{result1}]
For the diffusion process $T_u$, we meet the required conditions 
on the drift for $\kappa <8$ since the drift $\frac{-4/\kappa T}{1+T^2}=-\frac{1}{2}\frac{8/\kappa T}{1+T^2}$ and $\frac{-8/\kappa T}{1+T^2} \sim  \frac{8}{\kappa T}$ with $\frac{8}{\kappa}>1$, for $\kappa<8$.

For $n \in \mathbb N$ let us define $$a_n:=\log\left(1+\frac{4n}{\kappa}\right),$$
and let us consider

$$s_n(\omega)=\inf \left\lbrace t \in [0, \infty): \int_0^t\frac{ds}{y_s^2}(\omega)=a_n \right\rbrace.$$

Thus, in our case,  we take $t=a_n$,  and from the Lemma \ref{timechange}, we have that  $$\sup_{x \in K}|P_{a_n}(f(T_u))-\mu(f)| \leq C(K) \frac{1}{{a_n}^{\gamma}}||f||_{\infty} \leq C(K) \frac{1}{{(\log (1+4n/\kappa}))^{\gamma}}||f||_{\infty}.$$
Thus, $\sup_{x \in K}|P_{a_n}(f(T_u))-\mu(f)| \to 0$ as $n \to \infty\,.$

Finally, in terms of the backward Loewner differential equation, since by definition $T_{a_n}=\frac{X_{s_n(\omega)}(i)}{Y_{s_n(\omega)}(i)}$, we obtain that as $n \to \infty$, on the random sequence $s_n(\omega)$, for $f \in L^{\infty}(\mathbb R)$, the law of the limiting random variable $\lim_{n\to \infty}f\left(\frac{X_{s_n(\omega)}(i)}{Y_{s_n(\omega)}(i)}\right)$  converges to the stationary law of the process $T_u$ that is explicitly computed.
In particular, on this sequence of times, we obtained the  convergence in the total variation sense towards the uniform distribution for $\kappa=4$.
Thus, this method provides us with a solution to a Skorokhod type embedding problem for the backward Loewner differential equation, in the sense that it identifies a sequence of random times along the backward $SLE_{\kappa}$ flow started from $y_0=1$ on which the argument of points is uniformly distributed for $\kappa=4$.  In fact this is a closed-form expression family of distributions that are recovered using the stationary law of the diffusion process $T_u$ for $\kappa<8$.
\end{proof}

\subsection{Conjecture}
 

%

We consider the scaling of the $SLE_{\kappa}$ maps 
\begin{align}\label{sca}
\frac{h_{t}(i)}{\sqrt{t}}=h_{1}(i/\sqrt{t})
\end{align}
in distribution, as processes for $t \in [0, +\infty).$ \\
We can recover the same identity in distribution as processes for the shifted mappings also:

$$ \frac{h_{t}(i)-\sqrt{\kappa}B_{t}}{\sqrt{t}}=h_{1}(i/\sqrt{t})-\sqrt{\kappa}B_{1}\,.$$ 

Since the $SLE_{\kappa}$ trace is defined a.s., its tip is a well defined random variable. Then, the limit $\lim_{t \to \infty} ctg(arg(\hat{h}_1(i/\sqrt{t})))$ exists for almost every Brownian motion and is therefore a well-defined random variable.

Our main conjecture is the following:
\begin{conjecture}
For $\kappa <8$, the density of the law of the cotangent of the argument of the tip of the $SLE_{\kappa}$ trace at any fixed  deterministic time is 
$$\mu(dx)=\frac{C dx}{(1+x^2)^{4/\kappa}},$$ 
where $C$ is a normalizing constant.
\end{conjecture}
The conjecture is supported also by the phase transition at $\kappa=8$ that we described in the previous sections.

Using that  $\operatorname{Im}\hat{h}_{1}(0) >0$ a.s. (since the event that $\gamma_1 \in \mathbb{R}$ has zero probability), along with the continuity of the cotangent and argument in $\mathbb{C}$ at $z$ for $|z| \neq 0$,  we obtain that 
\begin{equation}\label{4.8}
\lim_{t \to \infty}ctg(arg(\hat{h}_{1}(i/\sqrt{t})))=ctg(arg(\lim_{t \to \infty}\hat{h}_{1}(i/\sqrt{t})))=ctg(arg(\hat{h}_{1}(0)))\,.
\end{equation}

In order to study the distribution at a fixed capacity time $t$ of $ctg(arg(\hat{h}_{t}(0)))$, we use the scaling (\ref{sca}) along with the fact that the law of $ctg(arg(\frac{\hat{h}_{t}(i)}{\sqrt{t}}))$  is the same as the law of the cotangent of  $ctg(arg(\hat{h}_{t}(i))$. Thus, we can identify the distribution of $\lim_{t \to \infty}ctg(arg(\hat{h}_{t}(i)))$ with the distribution of
$\lim_{t \to \infty}ctg(arg(\hat{h}_{t}(i/\sqrt{t}))).$

The $\lim_{t \to \infty}ctg(arg(\hat{h}_{t}(i)))$ does not exist a.s., however as a random variable we believe that its law converges to the stationary distribution of $T_u$ started from $T_0=0$.

\bibliographystyle{plain}
\bibliography{literaturegoodphasetransition}

\end{document}